\documentclass[11pt, reqno]{amsart}
\usepackage{latexsym}
\usepackage{color,dsfont}
\usepackage{xcolor}
\usepackage{amssymb,amsfonts,amsmath,mathrsfs}
\usepackage{cases}
\newcommand{\kommentar}[1]{}

\newtheorem{lem}{Lemma}[section]
\newtheorem{prop}[lem]{Proposition}
\newtheorem{thm}[lem]{Theorem}

\theoremstyle{definition}

\numberwithin{equation}{section}

\begin{document}

\title[Negative moments of the derivative of the zeta function]{Negative discrete moments of the derivative \\ of the Riemann zeta-function}

\author[H. M. Bui, A. Florea, and M. B. Milinovich]{Hung M. Bui, Alexandra Florea, and Micah B. Milinovich}
\address{Department of Mathematics, University of Manchester, Manchester M13 9PL, UK}
\email{hung.bui@manchester.ac.uk}
\address{Department of Mathematics, UC Irvine, Irvine, CA 92697, USA}
\email{floreaa@uci.edu}
\address{Department of Mathematics, University of Mississippi, University, MS 38677, USA}
\email{mbmilino@olemiss.edu}

\begin{abstract}
We obtain conditional upper bounds for negative discrete moments of the derivative of the Riemann zeta-function averaged over a subfamily of zeros of the zeta function which is expected to have full density inside the set of all zeros. For $k\leq 1/2$, our bounds for the $2k$-th moments are expected to be almost optimal. Assuming a conjecture about the maximum size of the argument of the zeta function on the critical line, we obtain upper bounds for these negative moments of the same strength while summing over a larger subfamily of zeta zeros.

\end{abstract}


\allowdisplaybreaks

\maketitle

\section{Introduction}
Let $\zeta(s)$ denote the Riemann zeta-function. In this paper, we study negative discrete moments of the derivative $\zeta(s)$ of the form
\[
J_{-k}(T) = \sum_{T < \gamma  \le 2T} |\zeta'(\rho)|^{-2k},
\]
where $k \in \mathbb{R}$ and $\rho=\beta+i\gamma$ runs over the non-trivial zeros of $\zeta(s)$. We will be interested in the case $k>0$, where we note that $J_{-k}(T)$ is only defined if the zeros $\rho$ are all simple. 

\smallskip

Gonek \cite{gonek} and Hejhal \cite{H} independently conjectured
that
\begin{equation}
J_{-k}(T) \asymp T (\log T)^{(k-1)^2},
\label{conj_jk}
\end{equation}
as $T \to \infty$ and for every real number $k$.  Using a random matrix theory model for $\zeta(s)$, Hughes, Keating, and O'Connell \cite{rmt_model} have refined this conjecture to predict that
$$J_{-k}(T) \sim C_k \, T  (\log T)^{(k-1)^2},$$ for a specific constant $C_k$, when $k \in \mathbb{C}$ and $\Re(k) < 3/2$. This conjecture was recovered by Bui, Gonek, and Milinovich \cite{BGM} using a hybrid Euler-Hadamard product for $\zeta(s)$. The random  matrix model for $\zeta(s)$, as well as unpublished work of Gonek, suggests that there are infinitely many zeros $\rho$ such that $|\zeta'(\rho)|^{-1} \gg |\gamma|^{1/3-\varepsilon}$, for any $\varepsilon>0$. This, in turn, would imply that for $k>0$, $J_{-k}(T) = \Omega(T^{2k/3-\varepsilon})$, suggesting that \eqref{conj_jk} only holds for $k<3/2$. However, it is noted in \cite{rmt_model} that while \eqref{conj_jk} is not expected to hold for $k \geq 3/2$, it is plausible that \eqref{conj_jk} could hold when summing over a subfamily of zeros  of full density inside the set of zeros with ordinates between $T$ and $2T$. Indeed, if one were to redefine $J_{-k}(T)$ to exclude a set of rare points where $|\zeta'(\rho)|$ is very close to $0$, then \eqref{conj_jk} should still predict the behavior of the redefined sum $J_{-k}(T)$. In Theorems \ref{mainthm} and \ref{thm_enlarged}, below, we make the first progress towards this idea, by obtaining upper bounds for $J_{-k}(T)$ over certain subfamilies of the nontrivial zeros of $\zeta(s)$ that expected to have full density inside the set of all nontrivial zeros of the zeta function.

\smallskip

Results towards Conjecture \eqref{conj_jk} are known in a few cases. Assuming the Riemann Hypothesis (RH), Gonek \cite{gonek_j1} has shown that $J_1(T) \sim \frac{1}{24 \pi}T  (\log T)^4$, while Ng \cite{ng} has shown that $J_2(T) \asymp T (\log T)^9$. Under RH and the assumption that the zeros of $\zeta(s)$ are all simple, Gonek \cite{gonek} has also shown that $J_{-1}(T) \gg T$ (see also \cite{ng_milinovich2}), which is a lower bound consistent with \eqref{conj_jk}. Furthermore, Gonek conjectured \cite{gonek} that
$$J_{-1}(T) \sim \frac{3}{\pi^3} T.$$
This agrees with the conjecture of Hughes, Keating, and O'Connell in the case $k=1$. 

\smallskip

While asymptotic formulas are known only in the cases mentioned above, progress has been made towards upper and lower bounds in certain cases. When $k<0$ (i.e., for positive discrete moments), Ng and Milinovich \cite{ng_milinovich} obtained sharp lower bounds consistent with \eqref{conj_jk} for integer $k$ under the Generalized Riemann Hypothesis (GRH) for Dirichlet $L$-functions. Milinovich \cite{mil_ub} obtained almost sharp upper bounds for positive discrete moments when $k$ is an integer under the RH, and Kirila \cite{kirila} extended this result to all $k<0$, obtaining sharp upper bounds of the order conjectured in \eqref{conj_jk}. These results settle, on GRH, conjecture \eqref{conj_jk} for negative integers $k$.

\smallskip

When $k>0$ (i.e., in the case of negative discrete moments), less is known. Heap, Li, and J.~Zhao \cite{heap_li_zhao} obtained lower bounds of the size in \eqref{conj_jk} for all fractional $k$ under RH and the assumption of simple zeros, while recently Gao and L.~Zhao \cite{gao_zhao} extended this result to all $k>0$ under the same assumptions. As mentioned above, these lower bounds are expected to be sharp for $k<3/2$, but not for $k \geq 3/2$.

\subsection{Main results} No upper bounds are known for $J_{-k}(T)$ in the case when $k>0$. The main results of this paper are to obtain upper bounds for the negative discrete moments when summing over two different subfamilies of the nontrivial zeros of $\zeta(s)$, both of which are expected to have full density inside the set of all zeros. Our first result assumes RH, while our second result assumes RH and a conjectural upper bound on $|S(t)|$, where $S(t)=\pi^{-1} \arg \zeta(1/2+it)$ and the argument is obtained by continuous variation along the straight line segments joining the points $2, 2+it$, and $1/2+it$ starting with the value $\arg\zeta(2)=0$. 

\smallskip

Let 
$$\mathcal{F} = \Big\{ \gamma \in (T,2T] : |\gamma-\gamma'| \gg \frac{1}{\log T}, \text{ for any other ordinate } \gamma'  \Big\}.$$
We expect $\mathcal{F}$ to have full density inside the set of zeros in $(T, 2T]$. Indeed,  for any fixed $\beta>0$, Montgomery's Pair Correlation Conjecture \cite{montgomery} implies that if $\gamma^{+}$ denotes the next largest ordinate of a nontrivial zero of $\zeta(s)$ above $\gamma$, then
$$ \# \Big\{  \gamma \in (T, 2T] : \, \gamma^{+} - \gamma \leq \frac{2 \pi \beta}{\log T}\Big\} \ll \beta^3 N(T),$$ where $N(T)$ denotes the number of nontrivial zeros of $\zeta(s)$ with ordinates $\gamma \in (0, T]$. Therefore, the set of excluded zeros whose ordinates do not belong to the family $\mathcal{F}$ conjecturally has measure $0$.

\smallskip

We will prove the following theorem.
\begin{thm}
\label{mainthm}
Assume the Riemann hypothesis and let $\varepsilon>0$. Then, for any $\delta>0$, we have that
\begin{equation}
\sum_{\gamma \in \mathcal{F}} |\zeta'(\rho)|^{-2k} \ll 
\begin{cases}
T^{1+\delta}, & \mbox{ if } \ 2k\,(1+\varepsilon)\leq 1, \\
T^{k+\frac{1}{2}+\delta}, & \mbox{ if } \ 2k \,(1+\varepsilon)>1.
\end{cases}
\label{to_prove}
\end{equation}
\end{thm}

The proof of Theorem \ref{mainthm} relies on Littlewood's classical estimate \cite{littlewood} that
\[
 |S(t)| \ll \frac{\log t}{\log\log t},
\]
assuming RH as $t\to\infty$. It is expected that maximum size of $|S(t)|$ is much smaller than this bound. 
In \cite{fgh}, Farmer, Gonek, and Hughes conjecture that 
\begin{equation}
 |S(t)| \ll \sqrt{ \log t \log \log t}.
 \label{conj_st}
\end{equation}
Now let $f(t)$ be any function slowly going to infinity and let
\begin{equation}\label{enlarged}
\mathcal{F}^{\text{enl}} = \Big\{ \gamma \in (T,2T] : |\gamma-\gamma'| \gg \frac{1}{ \exp \Big( \frac{ \sqrt{\log T}}{ f(T) \sqrt{\log \log T}} \Big)}, \text{ for any other ordinate } \gamma'  \Big\}.
\end{equation}
Then we will prove the following theorem.
\begin{thm}
\label{thm_enlarged}
 Let $\varepsilon>0$, and assume both the Riemann hypothesis and the conjecture in \eqref{conj_st}. Then, for any $\delta>0$, we have that
\begin{equation*}
\sum_{\gamma \in \mathcal{F}^{\text{enl}}} |\zeta'(\rho)|^{-2k} \ll 
\begin{cases}
T^{1+\delta}, & \mbox{ if } \ 2k\,(1+\varepsilon)\leq 1, \\
T^{k+\frac{1}{2}+\delta}, & \mbox{ if } \ 2k \,(1+\varepsilon)>1.
\end{cases}
\end{equation*}
\end{thm}

{\sc Remark.} Note that if we instead assume a conjectural bound for $S(t)$ of the form
$$|S(t)| \ll h(t)$$ for some function $h(t)$ such that $h(t) =o( \frac{\log t}{\log \log t})$, then we can define a family of zeros
$$\mathcal{F}^{'} = \Big\{ \gamma \in (T,2T] : |\gamma-\gamma'| \gg \frac{1}{ \exp \Big( \frac{ \log T}{ f(T) h(T)} \Big)}, \text{ for any other ordinate } \gamma'  \Big\},$$ where, as before, $f(t)$ is any function slowly going to infinity. Then, modifying our proof of Theorems \ref{mainthm} and \ref{thm_enlarged} in a straightforward manner, we can similarly prove that \eqref{to_prove} holds when summing over the zeros with $\gamma \in \mathcal{F}'$.
 
 \smallskip

\subsection{Bounds for $J_{-k}(T)$ assuming the Weak Mertens Conjecture}

If one were to sum over the full family of zeros of $\zeta(s)$, the problem of obtaining upper bounds for the negative discrete moments becomes more delicate. In order to obtain an upper bound, it seems that simply assuming RH and the simplicity of zeros is not enough. Indeed, from equation \eqref{tobound} in the proof of Lemma $2.1$, one sees that $\log|\zeta'(\rho)|$ is closely connected to a sum over the zeros $\rho'\neq \rho$, and in order to control $\log |\zeta'(\rho)|$ we need to understand the small differences $\gamma'-\gamma$. Hence, one would need to assume a lower bound on the size of the smallest difference between consecutive zeros of $\zeta(s)$. 

\smallskip

It is possible to obtain upper bounds on the negative discrete moments $J_{-k}(T)$ when summing over the full family of zeros of $\zeta(s)$ assuming a well-known conjecture for the distribution of partial sums of the M\"{o}bius function. Defining $M(x) =\sum_{n\leq x} \mu(n)$,  the Weak Mertens Conjecture (WMC) states that 
\[ 
\int_1^X \Big(  \frac{M(x)}{x} \Big)^2 \, dx \ll \log X.
\]
Titchmarsh \cite[Chapter~14]{titchmarsh} has shown that the WMC implies RH, that the zeros of $\zeta(s)$ are all simple, the fact that sum 
\begin{equation}
\label{conv}
\sum_{\rho} \frac{1}{|\rho \, \zeta'(\rho)|^2}
\end{equation} 
over all nontrivial zeros of $\zeta(s)$ is convergent. 
From these facts, we now observe that WMC implies that
\[
J_{-k}(T)  =
\begin{cases}
o\!\left( T^{k+1} (\log T)^{1-k} \right), & \mbox{ if } \ 0<k<1, \\
o\!\left( T^{2k}\right), & \mbox{ if } \ k\ge1.
\end{cases}
\]
Though these estimates apply to sums over all zeros, we note that both of these bounds are weaker than the bounds obtained in Theorem \ref{mainthm} for zeros restricted to the set $\mathcal{F}$.

\smallskip

If $k \geq 1$, then using Cauchy's inequality and WMC, we have
\begin{equation} \label{bd1}
\begin{split}
J_{-k}(T) = \sum_{T < \gamma  \le 2T} \frac{1}{| \zeta'(\rho)|^{2k}} & \leq \Big( \sum_{\gamma \in (T,2T]} \frac{1}{|\rho \, \zeta'(\rho)|^{2k}} \Big)^{1/2} \Big( \sum_{\gamma \in (T, 2T]} \frac{|\rho|^{2k}}{|\zeta'(\rho)|^{2k}} \Big)^{1/2} \\
& =o\Big(T^k J_{-k}(T)^{1/2}\Big),
\end{split}
\end{equation}
as $T\to\infty$, where we used the fact that the first sum on the right-hand side of the inequality in the first line is $o(1)$, which follows from \eqref{conv} and the fact that $k \geq 1$. This implies that $J_{-k}(T) = o\big(T^{2k}\big)$ for $k \geq 1$. On the other hand, if $k \in (0,1)$, then using H\"{o}lder's inequality, we have that 
\[
\begin{split}
J_{-k}(T) =\sum_{T < \gamma  \le 2T} \frac{1}{| \zeta'(\rho)|^{2k}} &\leq \Big( \sum_{\gamma \in (T,2T]} \frac{1}{|\zeta'(\rho)|^{2}} \Big)^{k}  \Big(  \sum_{\gamma \in (T, 2T]} 1 \Big)^{1-k} 
\\
&= O\Big( J_{-1}(T)^k (T \log T)^{1-k} \Big) = o\Big( T^{k+1} (\log T)^{1-k} \Big),
\end{split}
\]
as $T \to\infty$, which follows by using the bound $J_{-1}(T) = o(T^2)$ derived in \eqref{bd1} and the fact that $N(2T) - N(T) \ll T \log T$. 

\subsection{Overview of the paper}
To prove Theorems \ref{mainthm} and \ref{thm_enlarged}, we start by relating $\log |\zeta'(\rho)|$ to $\log|\zeta(\rho+1/\log T)|$ and a sum over the zeros $\rho' \neq \rho$, which we then bound using the lower bound on the difference between consecutive zeros provided by the families $\mathcal{F}$ and $\mathcal{F}^{\text{enl}}$. The problem of bounding the discrete negative moments of the derivative of $\zeta(s)$ then reduces to bounding negative discrete shifted moments of $\zeta(s)$. To do that, we build on work of the first two authors \cite{bf_negative} on upper bounds for the negative continuous moments of $\zeta(s)$, which in turn builds on work of Harper \cite{harper} and Soundararajan \cite{sound_ub} on upper bounds for the positive moments of $\zeta(s)$ on the critical line. Our analysis uses a discrete mean-value theorem for the mean-square of Dirichlet polynomials averaged over zeros of $\zeta(s)$, similar in flavor to the Montgomery-Vaughan's mean-value theorem for continuous averages Dirichlet polynomials \cite[Corollary 3]{MV}. The proof of this discrete mean-value estimate relies on the Landau-Gonek explicit formula \cite{gonek_landau}.

\smallskip

In Section \ref{first_steps}, we relate $\log |\zeta'(\rho)|$ to $\log |\zeta(\rho+1/\log T)|$ building upon previous observations of Kirila \cite{kirila}. Then we gather a few facts that we will need from \cite{bf_negative} and state the main propositions which evaluate moments of certain Dirichlet polynomials. In Section \ref{mvt}, we prove our discrete mean-value theorem for Dirichlet polynomials averaged over zeros of $\zeta(s)$, and in Section \ref{section_proofs_prop} we prove the main propositions stated in Section \ref{first_steps}. We finally prove Theorems \ref{mainthm} and \ref{thm_enlarged} in Section \ref{proofs}.

\section{First steps}
\label{first_steps}
We first express $\log |\zeta'(\rho)|$ in terms of $\log |\zeta(\rho+1/\log T)|$.
\begin{lem}
\label{prelim}
Assume RH. For $\gamma \in \mathcal{F}$, we have that 
\begin{equation*}
\log \frac{1}{|\zeta'(\rho)|} = \log \frac{1}{|\zeta(\rho+\tfrac{1}{\log T})|} + O\Big(  \frac{\log T}{\log \log T} \Big).
\end{equation*}
\end{lem}
\begin{proof}
Equation $3.3$ in Kirila \cite{kirila} states that 
\begin{align}
\log |\zeta'(\rho)| &= \log |\zeta(\sigma+i\gamma)| +\Big(\sigma-\frac{1}{2}\Big)\Big(\frac{\log T}{2}+O(1)\Big)-\log \Big|\sigma- \frac{1}{2}\Big| \nonumber \\
&\qquad-\frac{1}{2} \sum_{\rho' \neq \rho} \log \frac{ (\sigma-\frac{1}{2})^2+ (\gamma-\gamma')^2}{(\gamma-\gamma')^2},\label{tobound}
\end{align}
for any $\sigma>1/2$. We will obtain an upper bound for the sum over $\rho'$. Let $\sigma-1/2=\alpha$.  
We write
\begin{align*}
\sum_{\rho' \neq \rho} \log \frac{ \alpha^2+ (\gamma-\gamma')^2}{(\gamma-\gamma')^2} = \sum_{|\gamma'-\gamma| <\alpha} \log \frac{ \alpha^2+ (\gamma-\gamma')^2}{(\gamma-\gamma')^2}  + \sum_{|\gamma'-\gamma| \geq \alpha} \log \frac{ \alpha^2+ (\gamma-\gamma')^2}{(\gamma-\gamma')^2} .
\end{align*}
Let $M_1(\gamma)$ denote the first term on the right-hand side, and let  $M_2(\gamma)$ denote the second.

We pick $$\alpha = \frac{1}{\log T}.$$
We first bound $M_1(\gamma)$. We have 
\begin{equation}
M_1(\gamma)  \ll \sum_{|\gamma-\gamma'|<\frac{1}{\log T}} 1 \ll \frac{ \log T}{\log \log T},
\label{s1_bound}
\end{equation}
where we used the fact that $\gamma \in \mathcal{F}$, and that 
$$N\Big(\gamma+\frac{1}{\log T} \Big)  - N\Big( \gamma - \frac{1}{\log T} \Big) \ll 1+ \Big|S\Big(T+\frac{1}{\log T}\Big)- S \Big(T - \frac{1}{\log T}\Big)\Big| \ll \frac{\log T}{\log \log T}.$$
Standard estimates imply that
\begin{equation}
 M_2(\gamma) = \sum_{|\gamma-\gamma'| \geq \frac{1}{\log T}} \log \Big(1+ \frac{\alpha^2}{(\gamma-\gamma')^2} \Big)\ll \frac{\log T}{\log \log T}.
 \label{s2_bound}
 \end{equation} 
 Hence for $\gamma \in \mathcal{F}$, using \eqref{s1_bound} and \eqref{s2_bound}, we have that 
\begin{equation*}
\log \frac{1}{|\zeta'(\rho)|} = \log \frac{1}{|\zeta(\rho+1/\log T)|} + O\Big(  \frac{\log T}{\log \log T} \Big).
\end{equation*}
This completes the proof of the lemma.
\end{proof}

We similarly obtain the following estimate.
\begin{lem}
\label{enlarged_family}
For $\gamma \in \mathcal{F}^{\text{enl}}$, we have that
\begin{equation*}
\log \frac{1}{|\zeta'(\rho)|} = \log \frac{1}{|\zeta(\rho+\tfrac{1}{\log T})|} + O\Big(  \frac{\log T}{f(T)} \Big),
\end{equation*}
where $f$ is the function appearing in \eqref{enlarged}. 
\end{lem}
\begin{proof}
Upon noticing that 
$$M_1(\gamma) \ll \frac{\log T}{f(T)} \quad \text{ and } \quad M_2(\gamma) \ll \sqrt{ \log T \log \log T},$$ 
the proof is very similar to the proof of Lemma \ref{prelim}.
\end{proof}
Throughout the paper, we rely on the following lemma proved in \cite[Lemma 2.1 and (3.2)]{bf_negative}.
\begin{lem}
\label{lemma_lb}
Assume RH. Then we have that 
\begin{align*} &\log | \zeta(\tfrac12+\tfrac{1}{\log T}+i \gamma) |\geq \frac{ \log  \frac{ \gamma }{2 \pi } }{2 \pi \Delta} \log \big(1- e^{-2 \pi \Delta/\log T} \big)-\Re\Big( \sum_{p \leq e^{2 \pi \Delta}} \frac{b(p;\Delta)}{p^{1/2+1/\log T+i \gamma}}\Big)\\
	&\qquad\qquad - \frac{\log \log \log T}{2} - \Big( \log \frac{\log T}{\Delta} \Big)^{\eta(\Delta)} + O \Big(  \frac{\Delta^2 e^{\pi \Delta}}{1+\Delta T} + \frac{\Delta \log (1+ \Delta \sqrt{T})}{\sqrt{T}}+1\Big),
\end{align*}
where the coefficients $b(p;\Delta)$ can be written down explicitly, and 
$$\eta(\Delta)=\begin{cases}1 & \text{if }\Delta = o(\log T),\\ 0& \text{if }\Delta  \gg \log T.\end{cases}.$$
Moreover, the coefficients $b(p;\Delta)$ satisfy the bound
\begin{equation}
\label{improved_bd}
|b(p;\Delta)| \leq b(\Delta) \Big( \log \frac{\log T}{\Delta} \Big)^{\eta(\Delta)},
\end{equation} 
where, for any $\varepsilon>0$,
\begin{equation}
\label{bn}
b(\Delta) = 
\begin{cases}
\frac12+\varepsilon, & \mbox{ if } \Delta =o(\log T),\\
 \frac{1}{1-e^{-2\pi \Delta/\log T}}, & \mbox{ if } \Delta  \gg \log T.
\end{cases}
\end{equation}
\end{lem}

Now we introduce some of the notation in \cite{bf_negative}. Let 
\[
I_0 = (1, T^{\beta_0}],\ I_1 = (T^{\beta_0}, T^{\beta_1}],\ \ldots,\ I_K = (T^{\beta_{K-1}}, T^{\beta_K}]
\] 
for a sequence $\beta_0 , \ldots, \beta_K$ to be chosen later such that $\beta_{j+1}=r\beta_j$, for some $r>1$. Also, let $\ell_j$ be even parameters which we will also choose later on. Let $s_j$ be integers. For now, we can think of $s_j \beta_j \asymp 1$, and $\sum_{h=0}^K \ell_h \beta_h \ll 1$. We let $T^{\beta_j} = e^{2 \pi \Delta_j}$ for every $0\leq j \leq K$ and let 
\[
P_{u,v}(\gamma)=\sum_{p\in I_u} \frac{b(p;\Delta_v)}{p^{1/2+1/\log T+i\gamma}},
\] 
and we extend $b(p;\Delta)$ to a completely multiplicative function in the first variable.
%
Let 
$$ \mathcal{T}_u = \Big\{ \gamma \in (T,2T] : \max_{u \leq v \leq K} |P_{u,v}(\gamma)| \leq \frac{\ell_u}{ke^2}\Big\}.$$ 
Denote the set of $\gamma$ such that $\gamma \in \mathcal{T}_u$ for all $u \leq K$ by $\mathcal{T}'$. For $0 \leq j \leq K-1$, let $\mathcal{S}_j$ denote the subset of $\gamma \in (T,2T]$ such that $\gamma \in \mathcal{T}_h$ for all $h \leq j$, but $\gamma \notin \mathcal{T}_{j+1}$. 

\smallskip

Now for $\ell$ an integer, and for $z \in \mathbb{C}$, let
$$E_{\ell}(z) = \sum_{s \leq \ell} \frac{z^s}{s!}.$$
If $z \in \mathbb{R}$, then for $\ell$ even, $E_{\ell}(z) >0$. 
We will often use the fact that for $\ell$ an even integer and $z$ a complex number such that $|z| \leq \ell/e^2$, we have
$$e^{\Re(z)} \leq \max \Big\{ 1, |E_{\ell}(z)| \Big(1+\frac{1}{15e^{\ell}}\Big)\Big\},$$
whose proof can be found in \cite[Lemma 2.4]{bf_negative}.

\smallskip

Let $\nu(n)$ be the multiplicative function given by
$$\nu(p^j) = \frac{1}{j!},$$ for $p$ a prime and any integer $j$. We will frequently use the following fact: For any interval $I$, $s \in \mathbb{N}$ and $a(n)$ a completely multiplicative function, we have
\begin{equation}
 \Big(  \sum_{p \in I}  a(p) \Big)^s= s!  \sum_{\substack{ p |n \Rightarrow p \in I \\ \Omega(n) = s}} a(n) \nu(n),
 \label{identity}
 \end{equation} where $\Omega(n)$ denotes the number of prime factors of $n$, counting multiplicity.

\smallskip

We now state the following result, which is similar to \cite[Lemma 3.1]{bf_negative}.
\begin{lem}
\label{lem_initial}
For $\gamma \in (T,2T]$, we either have 
$$ \max_{0 \leq v \leq K} |P_{0,v}(\gamma)| > \frac{\ell_0}{ke^2},$$ or
$$|\zeta(\tfrac12+\tfrac{1}{\log T}+i\gamma)|^{-2k} \leq S_1(\gamma) +S_2(\gamma),$$ where 
\begin{align*}
S_1(\gamma)& =(\log \log T)^k\Big( \frac{1}{1-e^{-\beta_K}}\Big)^{ \frac{2k}{\beta_K } } \exp \Big( 2k \Big( \log \frac{\log T}{\Delta_K } \Big)^{\eta(\Delta_K)} \Big)  \\
&\qquad  \times  \prod_{h=0}^K \max \Big\{1, |E_{\ell_h}(kP_{h,K}(\gamma))|^2\Big(1+\frac{1}{15e^{\ell_h}}\Big)^2 \Big\}\\
&\qquad  \times \exp \Big( O \Big(  \frac{\Delta_K^2 e^{\pi \Delta_K}}{1+\Delta_K T} + \frac{\Delta_K \log (1+ \Delta_K \sqrt{T})}{\sqrt{T}}+1\Big)  \Big)
\end{align*}
and
\begin{align*}
S_2(\gamma) &=(\log \log T)^k \sum_{j=0}^{K-1}  \sum_{v=j+1}^K  \Big( \frac{1}{1-e^{-\beta_j}}\Big)^{ \frac{2k}{\beta_j } } \exp \Big( 2k \Big( \log \frac{\log T}{\Delta_j } \Big)^{\eta(\Delta_j)} \Big)  \\
&\qquad \times \prod_{h=0}^j\max \Big\{1, |E_{\ell_h}(kP_{h,j}(\gamma))|^2\Big(1+\frac{1}{15e^{\ell_h}}\Big)^2 \Big\}  \Big( \frac{ke^2}{\ell_{j+1}} |P_{j+1,v}(\gamma)|\Big)^{2s_{j+1}} \\
&\qquad \times  \exp \Big( O \Big(  \frac{\Delta_j^2 e^{\pi \Delta_j}}{1+\Delta_j T} + \frac{\Delta_j \log (1+ \Delta_j \sqrt{T})}{\sqrt{T}}+1\Big)  \Big).
\end{align*}
\end{lem}
\begin{proof}
The proof follows as the proof of \cite[Lemma 3.1]{bf_negative}, by noticing that for $\gamma \in (T,2T]$, either $\gamma \notin \mathcal{T}_0$, or $\gamma \in \mathcal{T}'$, or $\gamma \in \mathcal{S}_j$ for some $0 \leq j \leq K-1$, and then using Lemma \ref{lemma_lb}.
\end{proof}
We also have the following propositions, whose proofs we postpone to Section \ref{section_proofs_prop}. 
\begin{prop}
\label{sum0}
Assume RH. For $0 \leq v \leq K$ and $\beta_0 s_0 \leq 1- \log\log T/\log T$, we have
$$\sum_{\gamma \in (T,2T]} |P_{0,v}(\gamma)|^{2s_0} \ll N(T) s_0! b(\Delta_v)^{2s_0} \Big(  \log \frac{\log T}{\Delta_v } \Big)^{2s_0 \eta(\Delta_v)} ( \log \log T^{\beta_0})^{s_0}.$$
\end{prop}

\begin{prop}
\label{j+1}
Assume RH. Let $0 \leq j \leq K-1$. For $\beta_0$ as given in \eqref{s0} or \eqref{beta0_again}, and for $\sum_{h=0}^j \ell_h \beta_h+s_{j+1} \beta_{j+1} \leq 1- \log \log T/ \log T$ and $j+1\leq v \leq K$, we have
\begin{align*}
\sum_{\gamma \in (T,2T]} & \prod_{h=0}^j |E_{\ell_h}(k P_{h,j}(\gamma))|^2 |P_{j+1,v}(\gamma)|^{2s_{j+1}} \ll N(T) s_{j+1}! \big( \log T^{\beta_j} \big)^{k^2 b(\Delta_j)^2 (\log \frac{\log T}{\Delta_j })^{2 \eta(\Delta_j)}} \\
 & \times   b(\Delta_{j+1})^{2s_{j+1}} \Big( \log \frac{\log T}{\Delta_{j+1} } \Big)^{2s_{j+1} \eta(\Delta_{j+1})} (\log r)^{s_{j+1}}.
 \end{align*}
\end{prop}

\begin{prop}
\label{K_lemma}
Assume RH. For $\beta_0$ as given in \eqref{s0} or \eqref{beta0_again},  and for $\sum_{h=0}^K \ell_h \beta_h \leq 1-\log \log T/ \log T$, we have
\begin{align*}
\sum_{\gamma \in (T,2T]} & \prod_{h=0}^K |E_{\ell_h}(k P_{h,K}(\gamma))|^2   \ll N(T) \big( \log T^{\beta_K} \big)^{k^2 b(\Delta_K)^2 (\log \frac{\log T}{\Delta_K })^{2 \eta(\Delta_K)}} .
 \end{align*}
\end{prop}
\section{A mean-value theorem}
\label{mvt}
We now prove the following mean-value theorem, which is of independent interest.

\begin{thm} \label{mvt}
Assume RH
. Then for $x\ge 2$ and $T >1$, we have
\[
\begin{split}
\sum_{0<\gamma\le T} \Big| \sum_{n\le x} \frac{a_n}{n^{\rho}} \Big|^2 &= N(T) \sum_{n \le x} \frac{|a_n|^2}{n} - \frac{T}{\pi} \Re\Big( \sum_{kn \le x} \frac{\overline{a_{kn}} a_n \Lambda(k)}{kn}\Big) 
 + O\Big( x \, (\log (xT))^{2} \sum_{n \le x} \frac{|a_n|^2}{n} \Big),
\end{split}
\]
where $\{ a_n \}$ is any sequence of complex numbers. 
\end{thm}

The proof relies on following version of the Landau-Gonek explicit formula \cite[Theorem 1]{gonek_landau}. 
\begin{lem} \label{LG}
Let $y,T>1$.  Then 
\[
\begin{split}
\sum_{0 < \gamma \le T} y^{\rho} &= -\frac{T}{2\pi} \Lambda(y) + O\Big(y \log(2y T)\log\log(3y)\Big) 
\\
&\qquad \qquad + O\bigg( \log y  \min\Big\{ T, \frac{y}{\langle y \rangle}\Big\} \bigg)+ O\bigg( \log(2T) \min \Big\{T, \frac{1}{\log y} \Big\} \bigg),
\end{split}
\]
where $\langle y \rangle$ denotes the distance from $y$ to the nearest prime power other than $y$. Here, if $y\in \mathbb{N}$, $\Lambda(y)$ denotes the von Mangoldt function and $\Lambda(y)=0$, otherwise.
\end{lem}

\begin{proof}[Proof of Theorem \ref{mvt}] Assuming RH, we note that $1-\rho = \overline{\rho}$. So 
the sum on the left-hand side of expression in Theorem \ref{mvt} is equal to
\[
\begin{split}
\sum \limits_{0 < \gamma \leq T} \sum \limits_{m\le x} \frac{\overline{a_m}}{m^{1-\rho}} \sum_{n\le x} \frac{a_n}{n^{\rho}}=  N(T) \sum_{n \le x} \frac{|a_n|^2}{n} +   2\Re\bigg( \sum_{n\le x} \sum_{n<m\le x} \frac{\overline{a_m}a_n}{m}   \sum \limits_{0 < \gamma \leq T}  \left(\frac{m}{n}\right)^{\rho}\bigg).   
\end{split}
\]
%
%
%
Applying Lemma \ref{LG} with $y=m/n$, we have
\[
\begin{split}
\sum_{n\le x} & \sum_{n<m\le x} \frac{\overline{a_m}a_n}{m}   \sum \limits_{0 < \gamma \leq T}  \left(\frac{m}{n}\right)^{\rho} 
\\
 &= -\frac{T}{2\pi}  \sum_{n\le x} \sum_{n<m\le x}  \frac{ \overline{a_m} a_n \Lambda(m/n)}{m}  + O \Big( \log( xT) \log\log (3x) \sum_{m,n\le x}  \frac{|a_m a_n| }{n} \Big)
\\
&\qquad + O\Big( \sum_{n\le x} \sum_{n<m\le x}  \frac{|a_m a_n| \log(m/n) }{n \, \langle m/n \rangle }  \Big)
+ O\Big(\log (2T) \sum_{n\le x} \sum_{n<m\le x}  \frac{|a_m a_n| }{m \log(m/n)} \Big) 
\\
&= \Sigma_{1} + \Sigma_{2}+ \Sigma_{3}+ \Sigma_{4},
\end{split}
\]
say. In $\Sigma_{1}$, the only nonzero terms occur if $n$ divides $m$ (and $m/n$ is a prime power). Thus, making the substitution $m=nk$, it follows that
\[
\Sigma_1 =  -\frac{T}{2\pi}  \sum_{kn\le x} \frac{\overline{a_{kn}} a_n \Lambda(k)}{kn}.
\]

Next we estimate $\Sigma_2$. Observe that for any $\Delta>0$, we have
\begin{equation}\label{amgm}
|a_m a_n| \leq \frac{1}{2} \left(  \frac{|a_m|^2}{\Delta}+\Delta |a_n|^2 \right).
\end{equation}
This implies that 
\[
\begin{split}
\sum_{m,n\le x}  \frac{|a_m a_n|}{n} 
 &\ll   \frac{1}{\Delta} \sum_{m \le x } |a_m|^2 \sum_{n \le x }\frac{1 }{n}+\Delta \sum_{n\le x} \frac{|a_n|^2}{n} \sum_{m \le x } 1  
 \\
   &\ll    \frac{x \log x }{\Delta}  \sum_{m \le x } \frac{|a_m|^2}{m} +\Delta \, x \sum_{n\le x} \frac{|a_n|^2}{n}  \ll  x \, (\log x )^{1/2}  \sum_{n \le x } \frac{|a_n|^2}{n},
\end{split}
\]
upon choosing $\Delta =   (\log x )^{1/2}$. From this calculation, we see that 
\[
\begin{split}
\Sigma_2 
\ll x \, (\log (xT))^{3/2} \log\log (3x) \ \sum_{n \le x} \frac{|a_n|^2}{n}.
\end{split}
\]

We next estimate $\Sigma_3$. Again using \eqref{amgm}, for any $\Delta>0$, we have
\begin{equation} \label{sigma3}
\begin{split}
\Sigma_3 
&\ll \Delta \sum_{n\le x} |a_n|^2 \sum_{n<m \le x}  \frac{ \log(m/n) }{n \, \langle m/n \rangle } + \frac{1}{\Delta} \sum_{m \le x} |a_m|^2 \sum_{n<m}  \frac{ \log(m/n) }{n \, \langle m/n \rangle } = \Delta \Sigma_{31}+\frac{\Sigma_{32}}{\Delta},
\end{split}
\end{equation}
say. Write $m=qn+r$ with $-\frac{n}{2}<r \le \frac{n}{2}$. Then $\langle m/n \rangle = \langle q + \frac{r}{n} \rangle = \frac{|r|}{n}$ if $q$ is a prime power, and $\langle m/n \rangle \ge \frac{1}{2}$ if $q$ is not a prime power or $r=0$. Then
\[
\begin{split}
\Sigma_{31} &\ll \sum_{n\le x} |a_n|^2 \sum_{q \le \lfloor \frac{x}{n} \rfloor +1} \Lambda(q) \sum_{r \le n/2} \frac{1}{r} + \sum_{n\le x} \frac{|a_n|^2}{n}  \sum_{q \le \lfloor \frac{x}{n} \rfloor +1} \log(q+1)  \sum_{r \le n/2}1
\\
&\ll \log x\sum_{n\le x} |a_n|^2  \sum_{q \le \lfloor \frac{x}{n} \rfloor +1} \Lambda(q) + \sum_{n\le x} |a_n|^2 \sum_{q \le \lfloor \frac{x}{n} \rfloor +1} \log(q+1)  \ll x \log x \sum_{n\le x} \frac{|a_n|^2}{n}.
\end{split}
\]
Now again writing $m=qn+r$ with $-\frac{n}{2}<r \le \frac{n}{2}$ implies that $q|(m-r)$ and that
\[
 \frac{ \log(m/n) }{n \, \langle m/n \rangle } =  \frac{ \log (m/n) }{ n \, \langle q+r/n \rangle } \ll \frac{\log m}{|r|}
\]
if $q$ is a prime power, and that
\[
 \frac{ \log(m/n) }{n \, \langle m/n \rangle } \ll \frac{\log m}{n}
\]
if $q$ is not a prime power or $r=0$. The number of prime powers dividing $n$ equals $\Omega(n)$. Since $\Omega(n) \ll \log n$, 
 it follows that
\[
\begin{split}
\Sigma_{32}&\ll \log x \sum_{m \le x} |a_m|^2  \sum_{ \substack{ n<m \\ q n + r = m \\ -\frac{n}{2}<r \le \frac{n}{2}, r\ne0 \\ q \text{ prime power } }} \frac{1}{|r|}+ \log x\sum_{m \le x} |a_m|^2  \sum_{n<m} \frac{1}{n}
\\
&\ll \log x\sum_{m \le x} |a_m|^2  \sum_{\substack{-\frac{m}{2} < r \le \frac{m}{2} \\ r\ne 0}} \frac{1}{|r|} \sum_{\substack{q|(m-r) \\ q \text{ prime power}}} \!\!\! 1+ (\log x)^2\sum_{m \le x} |a_m|^2  
\\
&\ll \log x\sum_{m \le x} |a_m|^2  \sum_{\substack{-\frac{m}{2} < r \le \frac{m}{2} \\ r\ne 0}}  \frac{\Omega(m-r)}{|r|} + (\log x)^2\sum_{m \le x} |a_m|^2  
\\
&\ll (\log x)^2\sum_{m \le x} |a_m|^2  \sum_{r \le \frac{m}{2}}  \frac{1}{r}  \ll (\log x)^3\sum_{m \le x} |a_m|^2 \ll x (\log x)^3\sum_{m \le x}\frac{|a_m|^2}{m}.
\end{split}
\]
Choosing $\Delta = \log x $ in \eqref{sigma3}, we see that 
\[
\Sigma_3 \ll x (\log x)^{2} \sum_{n \le x} \frac{|a_n|^2}{n}.
\] 

Finally, we estimate $\Sigma_4.$ Note that if $m>n$, then $\log(m/n) = - \log\Big( 1 - \frac{m-n}{m} \Big) > \frac{m-n}{m},$
and it therefore follows that 
\[
\Sigma_4 \ll \log (2T) \sum_{n\le x} \sum_{n<m\le x}  \frac{|a_m a_n| }{m\!-\!n} \le \frac{\log( 2T)}{2}  \sum_{\substack{m,n\le x \\ m \ne n}}  \frac{|a_m a_n| }{|m\!-\!n|}.
\]
We note that
\[
\sum_{\substack{m,n\le x \\ m \ne n}}  \frac{|a_m a_n| }{|m\!-\!n|} \le  \frac{1}{2}  \sum_{\substack{m,n\le x \\ m \ne n}}  \frac{(|a_m|^2 + |a_n|^2) }{|m\!-\!n|} =  \sum_{n\le x} |a_n|^2 \sum_{\substack{m\le x \\ m \ne n}}  \frac{1 }{|m\!-\!n|}
\]
and
\[
\sum_{n\le x} |a_n|^2 \sum_{\substack{m\le x \\ m \ne n}}  \frac{1 }{|m\!-\!n|}  =  \sum_{n\le x} |a_n|^2 \left( \sum_{r \le n-1} \frac{1}{r} + \sum_{r \le x-n} \frac{1}{r} \right)  \ll \log x \sum_{n\le x} |a_n|^2.
\]
Hence
\[
\Sigma_4 \ll (\log (xT))^2  \sum_{n\le x} |a_n|^2 \ll x \, (\log (xT))^2  \sum_{n\le x} \frac{|a_n|^2}{n}.
\]

Combining estimates, the theorem follows. 
\end{proof}

\section{Proofs of Propositions \ref{sum0}, \ref{j+1}, and \ref{K_lemma}}
\label{section_proofs_prop}
Using Theorem \ref{mvt}, we are now ready to prove Propositions \ref{sum0}, \ref{j+1}, and \ref{K_lemma}.
\begin{proof}[Proof of Proposition \ref{sum0}]
Using \eqref{identity}, we have
$$P_{0,v}(\gamma)^{s_0} = s_0! \sum_{\substack{p|n \Rightarrow p \leq T^{\beta_0} \\ \Omega(n)=s_0}} \frac{b(n;\Delta_v) \nu(n)}{n^{1/2+1/\log T+i\gamma}}.$$
Then it follows that
\begin{align}
\sum_{\gamma \in (T,2T]} |P_{0,v}(\gamma)|^{2s_0} = (s_0!)^2 \sum_{\gamma \in (T,2T]} \Big| \sum_{\substack{p|n \Rightarrow p \leq T^{\beta_0} \\ \Omega(n)=s_0}} \frac{b(n;\Delta_v) \nu(n)}{n^{1/2+1/\log T+i\gamma}}  \Big|^2. \label{tbound}
\end{align}
Using Theorem \ref{mvt} twice for $0<\gamma\leq 2T$ and $0<\gamma \leq T$ and then differencing, we have that 
\begin{align*}
\eqref{tbound} &= (s_0!)^2 \bigg( \big(N(2T)-N(T)\big)  \sum_{\substack{p|n \Rightarrow p \leq T^{\beta_0} \\ \Omega(n)=s_0}} \frac{b(n;\Delta_v)^2 \nu(n)^2}{n^{1+2/\log T}} \\
&\qquad \qquad  - \frac{T}{\pi} \Re \Big(  \sum_{\substack{ p|kn \Rightarrow p \leq T^{\beta_0} \\ \Omega(kn)=s_0 \\ \Omega(n) =s_0}} \frac{ \overline{b(kn;\Delta_v)} b(n;\Delta_v) \nu(kn) \nu(n) \Lambda(k)}{k^{1+1/\log T} n^{1+2/\log T}} \Big)\\
&\qquad \qquad  +O \Big(T^{\beta_0s_0} (\log T)^2  \sum_{\substack{p|n \Rightarrow p \leq T^{\beta_0} \\ \Omega(n)=s_0}} \frac{b(n;\Delta_v)^2 \nu(n)^2}{n^{1+2/\log T}} \Big) \bigg).
\end{align*}
In the second sum above, note that we must have $k=1$, in which case the term vanishes. Then it follows that
\begin{align*}
\eqref{tbound} &= (s_0!)^2 \bigg( \Big(N(2T)-N(T) \Big) \sum_{\substack{p|n \Rightarrow p \leq T^{\beta_0} \\ \Omega(n)=s_0}} \frac{|b(n;\Delta_v)|^2 \nu(n)^2}{n^{1+2/\log T}} \\
&\qquad \qquad +O \Big(T^{\beta_0s_0} (\log T)^2  \sum_{\substack{p|n \Rightarrow p \leq T^{\beta_0} \\ \Omega(n)=s_0}} \frac{|b(n;\Delta_v)|^2 \nu(n)^2}{n^{1+2/\log T}} \Big) \bigg) \\
& \ll (s_0!)^2 \Big( N(T) + O \big( T^{\beta_0s_0} (\log T)^2 \big) \Big) \sum_{\substack{p|n \Rightarrow p \leq T^{\beta_0} \\ \Omega(n)=s_0}} \frac{|b(n;\Delta_v)|^2 \nu(n)^2}{n^{1+2/\log T}},
\end{align*}
where we have used the fact that $N(2T) \ll N(T)$.
Now using the condition that $\beta_0 s_0 \leq 1-\log \log T/\log T$ and  the bound $\nu(n)^2 \leq \nu(n)$, we get that
\begin{align*}
\eqref{tbound} & \ll  N(T) s_0! \Big(  \sum_{p \in I_0} \frac{|b(p;\Delta_v)|^2}{p^{1+2/\log T}}\Big)^{s_0} \\
& \ll N(T) s_0! b(\Delta_v)^{2s_0} \Big(  \log \frac{\log T}{\Delta_v } \Big)^{2s_0 \eta(\Delta_v)} ( \log \log T^{\beta_0})^{s_0},
\end{align*}
where we use the bound \eqref{improved_bd} to derive the estimate in the second line. This completes the proof.
\end{proof}

\begin{proof}[Proof of Proposition \ref{j+1}]
Using  \eqref{identity}, and similarly as in \cite{bf_negative}, we have
\begin{align*}
&\sum_{\gamma \in (T,2T]}  \prod_{h=0}^j |E_{\ell_h}(k P_{h,j}(\gamma))|^2 |P_{j+1,v}(\gamma)|^{2s_{j+1}} =(s_{j+1}!)^2 \\
&\  \ \times \bigg( \Big(  N(2T)-N(T) + O \big(T^{\sum_{h=0}^j \ell_h \beta_h + s_{j+1} \beta_{j+1}} (\log T)^2 \big)\Big) \\
&\qquad \times   \Big(\prod_{h=0}^j \sum_{\substack{p|n_h \Rightarrow p \in I_h \\ \Omega(n_h) \leq \ell_h}} \frac{|b(n_h;\Delta_j)|^2 k^{2 \Omega(n_h)} \nu(n_h)^2}{n_h^{1+2/\log T}} \Big) \sum_{\substack{p|n_{j+1} \Rightarrow p \in I_{j+1} \\ \Omega(n_{j+1})=s_{j+1}}}\frac{|b(n_{j+1};\Delta_v)|^2  \nu(n_{j+1})^2}{n_{j+1}^{1+2/\log T}} \\
&\quad-  \frac{T}{\pi} \Re\Big( \sum_{m=0}^j \sum_{\substack{p \in I_m \\ 1 \leq a \leq \ell_m}} \frac{ (\log p) \overline{b(p^a;\Delta_j)} k^a}{p^{a(1+1/\log T)}} \prod_{\substack{h=0 \\ h \neq m}}^j \sum_{\substack{p|n_h \Rightarrow p \in I_h\\ \Omega(n_h) \leq \ell_h}}\frac{ |b(n_h;\Delta_j)|^2 k^{2 \Omega(n_h)} \nu(n_h)^2}{n_h^{1+2/\log T}}\\
&\qquad \times \sum_{\substack{p | n_m \Rightarrow p \in I_m \\ \Omega(n_m) \leq \ell_m-a}} \frac{ |b(n_m; \Delta_j)|^2 k^{2 \Omega(n_m)} \nu(n_m) \nu(p^an_m)}{n_m^{1+2/\log T}} \sum_{\substack{p|n_{j+1} \Rightarrow p \in I_{j+1} \\ \Omega(n_{j+1})=s_{j+1}}}\frac{|b(n_{j+1};\Delta_v)|^2  \nu(n_{j+1})^2}{n_{j+1}^{1+2/\log T}}\Big) \bigg).
\end{align*}
 By \eqref{improved_bd} and the work in \cite{bf_negative} we get that
\begin{align}
 \sum_{\gamma \in (T,2T]} & \prod_{h=0}^j |E_{\ell_h}(k P_{h,j}(\gamma))|^2 |P_{j+1,v}(\gamma)|^{2s_{j+1}} \ll   N(T) s_{j+1}!  b(\Delta_{j+1})^{2s_{j+1}}  \Big( \log \frac{\log T}{\Delta_{j+1} } \Big)^{2s_{j+1}  \eta(\Delta_{j+1})} \nonumber  \\
 & \times (\log r)^{s_{j+1}} \bigg(  \big( \log T^{\beta_j} \big)^{k^2 b(\Delta_j)^2 (\log \frac{\log T}{\Delta_j })^{2 \eta(\Delta_j)}} + \sum_{m=0}^j \beta_m  \sum_{\substack{p \in I_m \\ 1\leq a \leq \ell_m}} \frac{ b(\Delta_j)^a (\log \frac{\log T}{\Delta_j } )^{a \eta(\Delta_j)}k^a}{p^{a(1+1/\log T)}} \nonumber  \\
 & \times \prod_{\substack{h=0 \\ h \neq m}}^j \sum_{\substack{p|n_h \Rightarrow p \in I_h\\ \Omega(n_h) \leq \ell_h }} \frac{ b(\Delta_j)^{2 \Omega(n_h)} (\log \frac{\log T}{\Delta_j })^{2 \eta(\Delta_j) \Omega(n_h)} k^{2 \Omega(n_h)} \nu(n_h)}{n_h^{1+2/\log T}}  \nonumber \\
& \times \sum_{\substack{p | n_m \Rightarrow p \in I_m \\ \Omega(n_m) \leq \ell_m-a}} \frac{ b(\Delta_j)^{2 \Omega(n_m)} (\log \frac{\log T}{\Delta_j })^{2 \eta(\Delta_j) \Omega(n_m)} k^{2 \Omega(n_m)} \nu(n_m) \nu(p^a n_m)}{n_m^{1+2/\log T}}\bigg). \label{tb3}
\end{align}

Now we use the fact that 
\begin{align}
\beta_m  \sum_{\substack{p \in I_m \\ 1\leq a \leq \ell_m}}& \frac{ b(\Delta_j)^a (\log \frac{\log T}{\Delta_j } )^{a \eta(\Delta_j)}k^a}{p^{a(1+1/\log T)}} \nonumber\\
&\qquad\times  \sum_{\substack{p | n_m \Rightarrow p \in I_m \\ \Omega(n_m) \leq \ell_m-a}} \frac{ b(\Delta_j)^{2 \Omega(n_m)} (\log \frac{\log T}{\Delta_j })^{2 \eta(\Delta_j) \Omega(n_m)} k^{2 \Omega(n_m)} \nu(n_m) \nu(p^a n_m)}{n_m^{1+2/\log T}} \nonumber  \\
& \ll \beta_m \sum_{\substack{p | n_m \Rightarrow p \in I_m \\ \Omega(n_m) \leq \ell_m}} \frac{ b(\Delta_j)^{2 \Omega(n_m)} (\log \frac{\log T}{\Delta_j })^{2 \eta(\Delta_j) \Omega(n_m)} k^{2 \Omega(n_m)} \nu(n_m)}{n_m^{1+2/\log T}}. \label{interm}
\end{align}
Indeed, if $m=0$ and if $b(\Delta_j )(\log \frac{\log T}{\Delta_j })^{\eta(\Delta_j)} k <1$, then 
$$ \sum_{\substack{p \in I_0 \\ 1\leq a \leq \ell_0}} \frac{ b(\Delta_j)^a (\log \frac{\log T}{\Delta_j } )^{a \eta(\Delta_j)}k^a}{p^{a(1+1/\log T)}} \ll \sum_{p \leq T^{\beta_0}} \frac{1}{p} \ll \log( \beta_0 \log T),$$
and with the choice \eqref{s0} and \eqref{beta0_again},  we get that
$$\beta_0  \sum_{\substack{p \in I_0 \\ 1\leq a \leq \ell_0}} \frac{ b(\Delta_j)^a (\log \frac{\log T}{\Delta_j } )^{a \eta(\Delta_j)}k^a}{p^{a(1+1/\log T)}} =O(1).$$
We also have that 
\begin{align*}
 \sum_{\substack{p | n_0 \Rightarrow p \in I_0 \\ \Omega(n_0) \leq \ell_0-a}}&  \frac{ b(\Delta_j)^{2 \Omega(n_0)} (\log \frac{\log T}{\Delta_j })^{2 \eta(\Delta_j) \Omega(n_0)} k^{2 \Omega(n_0)} \nu(n_0) \nu(p^a n_0)}{n_0^{1+2/\log T}} \\
 & \leq \sum_{\substack{p | n_0 \Rightarrow p \in I_0 \\ \Omega(n_0) \leq \ell_0}} \frac{ b(\Delta_j)^{2 \Omega(n_0)} (\log \frac{\log T}{\Delta_j })^{2 \eta(\Delta_j) \Omega(n_0)} k^{2 \Omega(n_0)} \nu(n_0)}{n_r^{1+2/\log T}},
\end{align*}
which shows \eqref{interm} in the case $m=0$. 

\smallskip

When $m \geq 1$, if $b(\Delta_j )(\log \frac{\log T}{\Delta_j })^{\eta(\Delta_j)} k <1$, then 
$$ \beta_m \sum_{\substack{p \in I_m \\ 1\leq a \leq \ell_m}} \frac{ b(\Delta_j)^a (\log \frac{\log T}{\Delta_j } )^{a \eta(\Delta_j)}k^a}{p^{a(1+1/\log T)}} \ll \sum_{p \in I_m} \frac{1}{p} =O(1),$$
and again
\begin{align*}
 \sum_{\substack{p | n_m \Rightarrow p \in I_m \\ \Omega(n_m) \leq \ell_m-a}}&  \frac{ b(\Delta_j)^{2 \Omega(n_m)} (\log \frac{\log T}{\Delta_j })^{2 \eta(\Delta_j) \Omega(n_m)} k^{2 \Omega(n_m)} \nu(n_m) \nu(p^a n_m)}{n_m^{1+2/\log T}} \\
 & \leq \sum_{\substack{p | n_m \Rightarrow p \in I_m \\ \Omega(n_m) \leq \ell_m}} \frac{ b(\Delta_j)^{2 \Omega(n_m)} (\log \frac{\log T}{\Delta_j })^{2 \eta(\Delta_j) \Omega(n_m)} k^{2 \Omega(n_m)} \nu(n_m)
 }{n_m^{1+2/\log T}}.
\end{align*}
The two bounds above establish \eqref{interm} when $b(\Delta_j )(\log \frac{\log T}{\Delta_j })^{\eta(\Delta_j)} k <1$.

On the other hand, if $b(\Delta_j )(\log \frac{\log T}{\Delta_j })^{\eta(\Delta_j)} k  \geq 1$, then we have that
$$ b(\Delta_j)^a (\log \frac{\log T}{\Delta_j } )^{a \eta(\Delta_j)}k^a \leq b(\Delta_j)^{2a} (\log \frac{\log T}{\Delta_j } )^{2a \eta(\Delta_j)}k^{2a},$$ and then
\begin{align*}
& \beta_m  \sum_{\substack{p \in I_m \\ 1\leq a \leq \ell_m}} \frac{ b(\Delta_j)^a (\log \frac{\log T}{\Delta_j } )^{a \eta(\Delta_j)}k^a}{p^{a(1+1/\log T)}} \\
&\qquad\qquad\times \sum_{\substack{p | n_m \Rightarrow p \in I_m \\ \Omega(n_m) \leq \ell_m-a}} \frac{ b(\Delta_j)^{2 \Omega(n_m)} (\log \frac{\log T}{\Delta_j })^{2 \eta(\Delta_j) \Omega(n_m)} k^{2 \Omega(n_m)} \nu(n_m) \nu(p^a n_m)}{n_m^{1+2/\log T}} \\
& \ll  \beta_m \sum_{\substack{p \in I_m \\ 1\leq a \leq \ell_m}} \frac{ b(\Delta_j)^{2a} (\log \frac{\log T}{\Delta_j } )^{2a \eta(\Delta_j)}k^{2a}}{p^{a(1+1/\log T)}} \\
&\qquad\qquad\times \sum_{\substack{p | n_m \Rightarrow p \in I_m \\ \Omega(n_m) \leq \ell_m-a}} \frac{ b(\Delta_j)^{2 \Omega(n_m)} (\log \frac{\log T}{\Delta_j })^{2 \eta(\Delta_j) \Omega(n_m)} k^{2 \Omega(n_m)} \nu(n_m) \nu(p^a n_m)}{n_m^{1+2/\log T}} \\
& \ll \beta_m\sum_{\substack{p | n_m \Rightarrow p \in I_m \\ \Omega(n_m) \leq \ell_m}} \frac{ b(\Delta_j)^{2 \Omega(n_m)} (\log \frac{\log T}{\Delta_j })^{2 \eta(\Delta_j) \Omega(n_m)} k^{2 \Omega(n_m)} \nu(n_m)}{n_m^{1+2/\log T}},
\end{align*}
which follows by rewriting $n_m \mapsto n_m p^a$ and noting that $p^{a/\log T} \ll 1$. Hence \eqref{interm} also follows in this case.

\smallskip

Now from \eqref{tb3} and \eqref{interm}, it follows that
\begin{align*}
\sum_{\gamma \in (T,2T]}&  \prod_{h=0}^j |E_{\ell_h}(k P_{h,j}(\gamma))|^2 |P_{j+1,v}(\gamma)|^{2s_{j+1}} \ll   N(T) s_{j+1}!  b(\Delta_{j+1})^{2s_{j+1}}  \Big( \log \frac{\log T}{\Delta_{j+1} } \Big)^{2s_{j+1}  \eta(\Delta_{j+1})} \nonumber  \\
 & \times (\log r)^{s_{j+1}} \bigg(  \big( \log T^{\beta_j} \big)^{k^2 b(\Delta_j)^2 (\log \frac{\log T}{\Delta_j })^{2 \eta(\Delta_j)}} \\
 &+ \prod_{h=0 }^j \sum_{\substack{p|n_h \Rightarrow p \in I_h\\ \Omega(n_h) \leq \ell_h }} \frac{ b(\Delta_j)^{2 \Omega(n_h)} (\log \frac{\log T}{\Delta_j })^{2 \eta(\Delta_j) \Omega(n_h)} k^{2 \Omega(n_h)} \nu(n_h)}{n_h^{1+2/\log T}} \bigg).
\end{align*}
Now using the work in \cite{bf_negative} to deal with the product over $h\leq j$, we get that
\begin{align*}
 \sum_{\gamma \in (T,2T]}  \prod_{h=0}^j |E_{\ell_h}(k P_{h,j}(\gamma))|^2 |P_{j+1,v}(\gamma)|^{2s_{j+1}}& \ll   N(T) s_{j+1}!  b(\Delta_{j+1})^{2s_{j+1}}  \Big( \log \frac{\log T}{\Delta_{j+1} } \Big)^{2s_{j+1}  \eta(\Delta_{j+1})} \nonumber  \\
 &\qquad \times (\log r)^{s_{j+1}}   \big( \log T^{\beta_j} \big)^{k^2 b(\Delta_j)^2 (\log \frac{\log T}{\Delta_j })^{2 \eta(\Delta_j)}}.
 \end{align*}
 This completes the proof of the proposition.
\end{proof}

The proof of Proposition \ref{K_lemma} is very similar to the proof of Proposition \ref{j+1}, so we leave the details to the interested reader.

\section{Proof of Theorems \ref{mainthm} and \ref{thm_enlarged}}
\label{proofs}
We now prove Theorem \ref{mainthm}. The proof of Theorem \ref{thm_enlarged} will follow in exactly the same way. 
\begin{proof}[Proof of Theorem \ref{mainthm}] If $2k(1+\varepsilon) \leq 1$, we choose
\begin{equation}
\beta_0 = \frac{a(2d-1) \log \log T}{(1+2\varepsilon) k \log T}, \, s_0 = \Big[ \frac{a}{\beta_0}\Big], \, \ell_0 =2 \Big\lceil \frac{s_0^d}{2} \Big\rceil, 
\label{s0}
\end{equation}
and
\begin{equation}
\beta_j = r^j \beta_0, \, s_j = \Big[\frac{a}{\beta_j} \Big], \, \ell_j = 2 \Big\lceil \frac{s_j^d}{2} \Big\rceil,
\label{betaj}
\end{equation}
where we can choose
$$ a= \frac{4-3k\varepsilon}{2(2-k\varepsilon)}, \quad  r = \frac{2}{2-k\varepsilon} , \quad
d = \frac{8-7k\varepsilon}{2(4-3k\varepsilon)},$$ 
so that
\begin{equation}
\frac{a(2d-1)}{r} = 1-k\varepsilon.
\label{adr}
\end{equation}
We choose $K$ such that
\begin{equation}
\label{betak}
\beta_K \leq c,
\end{equation}
for $c$ a small constant such that
\begin{equation}
c^{1-d} \Big( \frac{a^dr^{1-d} }{r^{1-d}-1}+\frac{2r}{r-1} \Big)\leq 1-a-\frac{\log \log T}{ \log T}.
\label{condition_c}
\end{equation}
The conditions above ensure that
$$\sum_{h=0}^j \ell_h \beta_h +s_{j+1} \beta_{j+1} \leq 1-\frac{\log \log T}{ \log T},$$ for any $j <K$, and that
$$\sum_{h=0}^K \ell_h \beta_h \leq 1-\frac{\log \log T}{ \log T}.$$

If $2k(1+\varepsilon)>1$, we choose the parameters as follows. 
Let
\begin{equation}
\beta_0 = \frac{(2k+2d-1- \frac{a(2d-1)}{r}) \log \log T }{(1+\delta)k \log T}, \quad s_0 =\Big [\frac{1}{\beta_0}\Big], \quad\ell_0 =2\Big\lceil\frac{ s_0^d}{2}\Big \rceil,
\label{beta0_again}
\end{equation}
where we pick
$$a = \frac{1-3k\varepsilon}{1-2k\varepsilon}, \quad r= \frac{1}{1-2k\varepsilon}, \quad d= \frac{2-7k\varepsilon}{2(1-3k\varepsilon)},$$
so that
\begin{equation}
\frac{a(2d-1)}{r} = 1-4k\varepsilon.
\label{adr2'}
\end{equation}
We choose $\beta_K$ as in \eqref{betak} and \eqref{condition_c}.

\smallskip

We have that
$$\sum_{\gamma \in \mathcal{F}} |\zeta'(\rho)|^{-2k} = \sum_{\substack{\gamma \in \mathcal{F} \\ \gamma \notin \mathcal{T}_0}}  |\zeta'(\rho)|^{-2k} + \sum_{\substack{\gamma \in \mathcal{F} \\ \gamma \in \mathcal{T}_0}}  |\zeta'(\rho)|^{-2k}.$$
Using Lemma \ref{prelim}, we have that for some $0 \leq v \leq K$,
\begin{align*}
\sum_{\substack{ \gamma \in \mathcal{F} \\ \gamma \notin \mathcal{T}_0}}&  |\zeta'(\rho)|^{-2k} \leq \exp \Big(O \Big(  \frac{\log T}{\log \log T}\Big) \Big) \sum_{\substack{\gamma \in \mathcal{F} \\ \gamma \notin \mathcal{T}_0}} |\zeta(1/2+1/\log T+i\gamma)|^{-2k} \\
& \leq \exp \Big(O \Big(  \frac{\log T}{\log \log T}\Big) \Big) \sum_{\gamma \in (T,2T]} |\zeta(1/2+1/\log T+i\gamma)|^{-2k} \Big( \frac{ke^2}{\ell_0} |P_{0,v}(\gamma)| \Big)^{2s_0} \\
& \leq \exp \Big(O \Big(  \frac{\log T}{\log \log T}\Big) \Big)\Big( \frac{ke^2}{\ell_0}\Big)^{2s_0} \Big( \frac{1}{1-(\log T)^{-2/\log T}} \Big)^{\frac{(1+\varepsilon) k \log T}{\log \log T}} \sum_{\gamma \in (T,2T]} |P_{0,v}(\gamma)|^{2s_0},
\end{align*}
where we used the pointwise bound 
\begin{equation}
 |\zeta(1/2+1/\log T+i\gamma)|^{-1} \leq \Big( \frac{1}{1-(\log T)^{-2/\log T}} \Big)^{\frac{(1+\varepsilon) \log T}{2 \log \log T}},
 \label{pointwise}
 \end{equation} which is proven in \cite[Lemma 2.2]{bf_negative}.
Using Proposition \ref{sum0}, we have that the above is
\begin{align*}
\ll   \exp \Big(O \Big(  \frac{\log T}{\log \log T}\Big) \Big) N(T) s_0! b(\Delta_v)^{2s_0} \Big(  \log \frac{\log T}{\Delta_v } \Big)^{2s_0 \eta(\Delta_v)} ( \log \log T^{\beta_0})^{s_0} \Big(\frac{ke^2}{\ell_0}\Big)^{2s_0} T^{(1+\varepsilon)k} .
\end{align*}
Now similarly as in \cite{bf_negative}, using Stirling's formula, we have that the above is 
\begin{align*}
\ll \exp \Big(O \Big(  \frac{\log T}{\log \log T}\Big) \Big)N(T)& T^{(1+\varepsilon)k} \sqrt{s_0}\exp\bigg(-(2d-1)s_0 \log s_0\\
& +2s_0 \log \Big(ke^{3/2} b(\Delta_0) \Big(\log \frac{\log T}{\Delta_0 }\Big)^{\eta(\Delta_0)} \sqrt{\log  \log T^{\beta_0}}  \Big) \bigg)  .
\end{align*}

If $2k(1+\varepsilon) \leq 1$ then using the choice of parameters \eqref{s0} and the fact that $\eta(\Delta_0)=1$, we get that
\begin{equation}
\sum_{\substack{ \gamma \in \mathcal{F} \\ \gamma \notin \mathcal{T}_0}} |\zeta'(\rho)|^{-2k} = o(N(T)).
\label{t0first}
\end{equation}
If $2k(1+\varepsilon) >1$, then using the choice \eqref{beta0_again}, again the fact that $\eta(\Delta_0)=1$, and the pointwise bound \eqref{pointwise} (with $\varepsilon$ replaced by $\delta$ on the right-hand side of the inequality) we get that 
 \begin{equation}
 \sum_{\substack{ \gamma \in \mathcal{F} \\ \gamma \notin \mathcal{T}_0}} |\zeta'(\rho)|^{-2k}  \ll T^{1+(1+\delta)k \frac{2k-\frac{a(2d-1)}{r}}{2k-\frac{a(2d-1)}{r}+2d-1}} \exp \Big( \frac{\log T \log \log \log T}{\log \log T} \Big).
 \label{t0_second}
 \end{equation}
 
 \smallskip

Now assume that $\gamma \in \mathcal{T}_0$. Using Lemma \ref{lem_initial}, we have that
\begin{equation}
\sum_{\substack{\gamma \in \mathcal{F} \\ \gamma \in \mathcal{T}_0}} |\zeta'(\rho)|^{-2k} \leq   \exp \Big(O \Big(  \frac{\log T}{\log \log T}\Big) \Big)\Big(  \sum_{\gamma \in \mathcal{F}} S_1(\gamma) + \sum_{\gamma \in \mathcal{F}} S_2(\gamma) \Big).
\label{sumt0}
\end{equation}
Using again Lemma \ref{lem_initial}, the choice \eqref{betak} for $\beta_K$, and noting that $\eta(\Delta_K)=0$, we have that
\begin{align*}
\sum_{\gamma \in \mathcal{F}} & S_1(\gamma) \ll  (\log \log T)^k    \sum_{\gamma \in (T,2T]} \prod_{h=0}^K \max \Big\{1, |E_{\ell_h}(kP_{h,K}(\gamma))|^2\Big(1+\frac{1}{15e^{\ell_h}}\Big)^2 \Big\} .
\end{align*}
Note that in the inequality above, we can assume without loss of generality that 
$$\max \Big\{1, |E_{\ell_h}(kP_{h,K}(\gamma))|^2\Big(1+\frac{1}{15e^{\ell_h}}\Big)^2 \Big\}=  |E_{\ell_h}(kP_{h,K}(\gamma))|^2\Big(1+\frac{1}{15e^{\ell_h}}\Big)^2.$$ 
Using Proposition \ref{K_lemma} it follows that
\begin{equation}
\sum_{\gamma \in \mathcal{F}} S_1(\gamma) \ll N(T) (\log T)^{O(1)},
\label{s1bound}
\end{equation}
in both cases $2k(1+\varepsilon) \leq 1$ and $2k(1+\varepsilon)>1$. 

\smallskip

Now we evaluate the contribution from $S_2(\gamma)$. Similarly to equation $4.9$ in \cite{bf_negative}, by using Proposition \ref{j+1}   we have that 
\begin{align}
&\sum_{\gamma \in \mathcal{F}} S_2(\gamma)   \ll N(T) (\log \log T)^k \sum_{j=0}^{K-1} (K-j) \sqrt{\frac{1}{\beta_{j+1}}} \exp \bigg( \frac{ \log \log T}{\beta_j} \Big( 2k-\frac{a(2d-1)}{r}  \Big) \nonumber \\
&\ \ + \frac{\log(\beta_j \log T)}{\beta_j} \Big( \frac{a(2d-1)}{r}-2k \Big)+ \frac{2a}{r \beta_j} \log \Big(ke^{3/2} b(\Delta_{j+1}) \Big(\log \frac{\log T}{\Delta_{j+1} }\Big)^{\eta(\Delta_{j+1})}\Big)  \nonumber\\\nonumber \\
&\  \ +\frac{2k}{\beta_j} \log \frac{1}{1-c/2} +\frac{a(2d-1)}{r \beta_j}\log\frac{r}{a} + 2k \Big( \log \frac{\log T}{\Delta_j } \Big)^{\eta(\Delta_j)} \bigg) \big( \log T^{\beta_j} \big)^{k^2 b(\Delta_j)^2 (\log \frac{\log T}{\Delta_j })^{2 \eta(\Delta_j)}}. \nonumber
\end{align}
Rearranging, we have that
\begin{align}
&\sum_{\gamma \in \mathcal{F}} S_2(\gamma)   \ll N(T) (\log \log T)^k \sum_{j=0}^{K-1} (K-j) \sqrt{ \frac{1}{\beta_{j+1}}} \exp \bigg(   \frac{\log \beta_j}{\beta_j} \Big( \frac{a(2d-1)}{r}-2k \Big) \nonumber  \\
& + \frac{2}{ \beta_j}  \log \Big(  \log \frac{\log T}{\Delta_{j+1} }\Big)^{\eta(\Delta_{j+1})}  + 2k \Big( \log \frac{\log T}{\Delta_j } \Big)^{\eta(\Delta_j)} + k^2 b(\Delta_j)^2 \Big(\log \frac{1}{\Delta_j \alpha}\Big)^{2 \eta(\Delta_j)} \log (\beta_j \log T) \nonumber \\
& +O \Big( \frac{1}{\beta_j} \Big) \bigg).\label{s22}
\end{align}
First assume that $2k(1+\varepsilon) \leq 1$. Then note that
$$2k - \frac{a(2d-1)}{r} \leq -k \varepsilon.$$
We first consider the contribution from those $j$ for which $\Delta_j =o(\log T)$, i.e., those $j$ for which $\beta_j \to 0$, in which case $\eta(\Delta_j)=1$. Let $R_1$ denote this contribution. Note that the first term inside the exponential is negative, so we have 
\begin{align*}
R_1 & \ll T(\log \log T)^{k+1} \sum_{j=0}^{K-1} \exp \Big( k^2 \Big(\frac{1}{4}+\varepsilon \Big) \Big(\log \frac{1}{\beta_j }\Big) \log (\beta_j \log T)  +O \Big( \log \frac{1}{\beta_j} \Big) \Big).
\end{align*}
Since $\beta_j \geq \beta_0$ and $\beta_0 \asymp \frac{\log \log T}{\log T}$ (see the choice \eqref{s0}), we get that
$$R_1 \ll N(T)(\log \log T)^{k+1} \sum_{j=0}^{K-1} \exp \Big( k^2 \Big(\frac{1}{4}+\varepsilon \Big) \Big(\log \frac{1}{\beta_j }\Big) \log (\beta_j \log T) +O (\log \log T) \Big).$$
Now if we let 
$$f(x) = \Big(\log \frac{1}{x} \Big) \log (x \log T),$$ we note that $f(x)$ attains its maximum at $x= \frac{1}{\log T}$, and then
\begin{equation}
R_1 \ll N(T) \exp \Big( k^2 (\log \log T)^2 \Big) \ll T^{1+\delta}.
\label{r1_bound}
\end{equation}
We now consider the contribution from those $j$ in \eqref{s22} for which $\eta(\Delta_j)=0$, i.e., those $j$ for which $\beta_j \gg 1$. It is easy to see that in this case, we have
\begin{equation}
R_2 \ll N(T) (\log T)^{O(1)}.
\label{r2_bound}
\end{equation}
Combining the bounds \eqref{r1_bound} and \eqref{r2_bound} shows that
\begin{equation}
\sum_{\gamma \in \mathcal{F}} S_2(\gamma) \ll T^{1+\delta},
\label{bound_s2_smallk}
\end{equation}
when $2k(1+\varepsilon) \leq 1$. Combining the bounds \eqref{bound_s2_smallk}, \eqref{s1bound}, and \eqref{t0first} proves Theorem \ref{mainthm} in the case $2k(1+\varepsilon)\leq 1$.

\smallskip

We now consider the case when $2k(1+\varepsilon) >1$. From equation \eqref{s22}, we get that
\begin{align*}
\sum_{\gamma \in \mathcal{F}} S_2(\gamma) \ll T(\log \log T)^{k+1} \sum_{j=0}^{K-1} \exp \Big(  \frac{ \log(1/\beta_j)}{\beta_j} \Big(2k - \frac{a(2d-1)}{r}  \Big)+ O \Big( \frac{ \log T \log \log \log T}{\log \log T} \Big)\Big).
\end{align*}
Now notice that since
$$2k - \frac{a(2d-1)}{r} > \frac{\varepsilon}{2},$$ the term inside the exponential is decreasing as a function of $j$, and hence it attains its maximum at $\beta_0$. Using the expression \eqref{beta0_again}, we get that 
\begin{equation}
\sum_{\gamma \in \mathcal{F}} S_2(\gamma) \ll T^{1+(1+\delta)k \frac{2k-\frac{a(2d-1)}{r}}{2k-\frac{a(2d-1)}{r}+2d-1}} \exp \Big( \frac{\log T \log \log \log T}{\log \log T} \Big).
\label{s2_bigk}
\end{equation}
Now combining the bounds \eqref{s2_bigk}, \eqref{s1bound}, and \eqref{t0_second} and after a relabeling of the $\varepsilon, \delta$, Theorem \ref{mainthm} in the case $2k(1+\varepsilon)>1$ follows.
\end{proof}

The proof of Theorem \ref{thm_enlarged} follows in exactly the same way, the only difference being the use of Lemma \ref{enlarged_family} instead of Lemma \ref{prelim}.
\bibliographystyle{amsalpha}

\section*{Acknowledgments}
AF was supported by the NSF grant DMS-2101769 and MBM was supported by the NSF grant DMS-2101912.

\end{document}